\documentclass[12pt]{amsart}
\usepackage{amssymb}
\usepackage{amsmath}
\usepackage{amsfonts}
\usepackage{graphicx}

plus 6pt minus 12pt
\newtheorem{thm}{Theorem}[section]
\newtheorem{cor}[thm]{Corollary}
\newtheorem{lem}[thm]{Lemma}

\newtheorem{prop}[thm]{Proposition}

\newtheorem{ex}[thm]{Example}
\newtheorem{prob}[thm]{Problem}

\newtheorem{defn}[thm]{Definition}

\numberwithin{equation}{section}

\usepackage{amsxtra}
\usepackage{eucal}
\usepackage{mathrsfs}
\usepackage{longtable}
\usepackage{caption}

\usepackage{hyperref}
    \usepackage{aeguill}
    \usepackage{type1cm}

\newcommand{\N}{\mathbb{N}}

\newcommand{\R}{\mathbb{R}}

\usepackage[usenames,dvipsnames]{color}

\usepackage[dvipsnames]{xcolor}

\begin{document}

\title{Lusin-type properties of convex functions and convex bodies}

\author{Daniel Azagra}
\address{Departamento de An{\'a}lisis Matem{\'a}tico y Matem\'atica Aplicada,
Facultad Ciencias Matem{\'a}ticas, Universidad Complutense, 28040, Madrid, Spain.  
}
\email{azagra@mat.ucm.es}

\author{Piotr Haj\l asz}
\address{Department of Mathematics,
University of Pittsburgh,
301 Thackeray Hall,
Pittsburgh, PA 15260, USA.}
\email{hajlasz@pitt.edu}

\date{November 11, 2020}



\begin{abstract}
We prove that if $f:\R^n\to\R$ is convex and $A\subset\R^n$ has finite measure, then for any $\varepsilon>0$ there is a convex function $g:\R^n\to\R$ of class $C^{1,1}$ such that $\mathcal{L}^n(\{x\in A:\,  f(x)\neq g(x)\})<\varepsilon$.
As an application we deduce that if $W\subset\R^n$ is a compact convex body then, for every $\varepsilon>0$, there exists a convex body $W_{\varepsilon}$ of class $C^{1,1}$ such that
$\mathcal{H}^{n-1}\left(\partial W\setminus \partial W_{\varepsilon}\right)< \varepsilon$.
We also show that if $f:\R^n\to\R$ is a convex function and $f$ is not  of class $C^{1,1}_{\rm loc}$, then for any $\varepsilon>0$ there is a convex function $g:\R^n\to\R$ of class $C^{1,1}_{\rm loc}$ such that
$\mathcal{L}^n(\{x\in \R^n:\,  f(x)\neq g(x)\})<\varepsilon$ if and only if $f$ is essentially coercive, meaning that  $\lim_{|x|\to\infty}f(x)-\ell(x)=\infty$ for some linear function $\ell$. A consequence of this result is that, if $S$ is the boundary of some convex set with nonempty interior (not necessarily bounded) in $\R^n$ and $S$ does not contain any line, then for every $\varepsilon>0$ there exists a convex hypersurface $S_{\varepsilon}$ of class $C^{1,1}_{\textrm{loc}}$ such that $\mathcal{H}^{n-1}(S\setminus S_{\varepsilon})<\varepsilon$.
\end{abstract}

\maketitle

\section{Introduction and main results}

Let $\mathcal{A}$ and $\mathcal{C}$ be two classes of real valued functions defined on $\R^n$ (or on an open subset of $\R^n$).  If for a given $f\in\mathcal{A}$ and every $\varepsilon>0$  we can find a function $g\in\mathcal{C}$ such that 
\begin{equation}\label{Lusin}
\mathcal{L}^n\left(\{x : f(x)\neq g(x)\}\right)<\varepsilon,
\end{equation} 
we say that $f$ {\em has the Lusin property of class $\mathcal{C}$.} Here, and in what follows $\mathcal{L}^{n}$ denotes the Lebesgue measure in $\R^n$. If every function $f\in\mathcal{A}$ satisfies this property, we also say that $\mathcal{A}$ has the Lusin property of class $\mathcal{C}$.

This terminology comes from the well known theorem of Lusin which asserts 
that for every Lebesgue measurable function $f:\R^n\to\R$ and every $\varepsilon>0$ there exists a continuous function $g:\R^n\to\R$ such that $\mathcal{L}^n\left(\{x : f(x)\neq g(x)\}\right)<\varepsilon$. That is, measurable functions have the Lusin property of class $C(\R^n)$. 

Several authors have shown that one can take $g$ of class $C^k$ if $f$ has some weaker regularity properties of order $k$. 
For instance, Federer \cite[p. 442]{Federer} showed that almost everywhere differentiable functions (and in particular locally Lipschitz functions) have the Lusin property of class $C^1$. Whitney \cite{Whitney2} improved this result by showing that a function $f:\R^n\to\R$ has approximate partial derivatives of first order a.e. if and only if $f$ has the Lusin property of class $C^1$; see also \cite{Alberti,MP} for related results.

In \cite[Theorem 13]{CalderonZygmund} Calder\'on and Zygmund proved analogous results for
$\mathcal{A}=W^{k,p}(\R^n)$ (the class of Sobolev functions) and $\mathcal{C}=C^k(\R^n)$. Other authors, including Liu \cite{Liu1977}, Bagby, Michael and Ziemer \cite{BagbyZiemer, MichaelZiemer, Ziemer}, Bojarski, Haj\l asz and Strzelecki \cite{BojarskiHajlasz, BojarskiHajlaszStrzelecki}, and Bourgain, Korobkov and Kristensen \cite{BourKoKris2} improved Calder\'on and Zygmund's result in several directions, by obtaining additional estimates for $f-g$ in the Sobolev norms, as well as the Bessel capacities or the Hausdorff contents of the exceptional sets where $f\neq g$. In \cite{BourKoKris2} some Lusin properties of the class $BV_{k}(\R^n)$ (of integrable functions whose distributional derivatives of order up to $k$ are Radon measures) are also established. On the other hand, generalizing Whitney's result \cite{Whitney2} to higher orders of differentiability, Isakov \cite{Isakov} and Liu and Tai \cite{LiuTai} independently established that a function $f:\R^n\to\R$ has the Lusin property of class $C^{k}$ if and only if $f$ is approximately differentiable of order $k$ almost everywhere (and if and only if $f$ has an approximate $(k-1)$-Taylor polynomial at almost every point). See also \cite{Francos,Whitney2} for related results.

The Whitney extension technique \cite{Whitney}, or other related techniques, such as the Whitney smoothing \cite{BojarskiHajlaszStrzelecki}, play a key role in the proofs of all of these results. 

For the special class of {\em convex} functions $f:\R^n\to\R$, Alberti, Imomkulov \cite{Alberti2, Imomkulov} and Evans and Gangbo \cite[Proposition~A.1]{EvansGangbo}
showed that every convex function has the Lusin property of class $C^2$; however, given a convex function $f$, the function $g\in C^2$ satisfying \eqref{Lusin} that they obtained is not necessarily convex. 
Indeed, their arguments were based on the Whitney extension theorem and 
Whitney's construction does not preserve convexity.
Thus it is natural to consider the following problems.

Denote by $C^{1,1}(\R^n)$, $C^{1,1}_\textrm{loc}(\R^n)$ and $C_{\textrm{conv}}(\R^n)$ respectively the class of real valued $C^1$ functions with Lipschitz continuous gradient, the class of $C^1$ functions with locally Lipschitz continuous gradient and the class of convex functions, and define
\begin{eqnarray*}
& C^{1,1}_{\textrm{conv}}(\R^n)=C^{1,1}(\R^n)\cap C_{\textrm{conv}}(\R^n); \\
& C^{1,1 \, \textrm{loc}}_{\textrm{conv}}(\R^n)=C^{1,1}_{\textrm{loc}}(\R^n)\cap C_{\textrm{conv}}(\R^n); \\
& C^2_{\textrm{conv}}(\R^n)=C^{2}(\R^n)\cap C_{\textrm{conv}}(\R^n).
\end{eqnarray*}

\begin{prob}\label{problem 1}
{\em Given $\mathcal{C}\in\{C^{1,1}_{\textrm{conv}}(\R^n), \, C^{1,1 \, \textrm{loc}}_{\textrm{conv}}(\R^n), \,
C^2_{\textrm{conv}}(\R^n)\}$, does $C_{\textrm{conv}}(\R^n)$ have the Lusin property of class $\mathcal{C}$?}
\end{prob}
Example~\ref{C11 does not suit unbounded domains} and Proposition~\ref{if f is not coercive nothing can be done globally} show that the answer to this question is in the negative. Thus we ask:

\begin{prob}\label{problem 2}
{\em For every such $\mathcal{C}$, can we characterize the subclass of $C_{\textrm{conv}}(\R^n)$ which has the Lusin property of class $\mathcal{C}$?}
\end{prob}

In Theorem~\ref{main theorem for loc C11 convex} we provide a complete answer to this question when $\mathcal{C}= C^{1,1 \, \textrm{loc}}_{\textrm{conv}}(\R^n)$.

\begin{prob}\label{problem 3}
{\em  What happens if we replace $\R^n$ with an open bounded convex subset $\Omega$ of $\R^n$?}
\end{prob}

In Corollary~\ref{main corollary for C11 convex} we show that the answer is in the positive when $\mathcal{C}=C^{1,1}_{\textrm{conv}}(\Omega)$ (and hence when  
$\mathcal{C}=C^{1,1 \, \textrm{loc}}_{\textrm{conv}}(\Omega)$). See also Theorem~\ref{main theorem for C11 convex}.

Our proofs are based on some results and techniques concerning $C^{1,1}_{\textrm{conv}}(\R^n)$ and $C^{1,1 \, \textrm{loc}}_{\textrm{conv}}(\R^n)$
extensions of $1$-jets recently obtained in \cite{Azagra2019,AzagraLeGruyerMudarra,AzagraMudarra1,  AzagraMudarra2}. Problems \ref{problem 2} and \ref{problem 3} remain open for $\mathcal{C}=C^{2}_{\textrm{conv}}(\R^n)$, and they look rather hard in the absence of a characterization of the $2$-jets (defined on an arbitrary subset of $\R^n$) which admit $C^{2}_{\textrm{conv}}(\R^n)$ extensions.

As an application of our results, we will show that all boundaries of compact convex bodies in $\R^n$ have the Lusin property of class $C^{1,1}_{\textrm{conv}}$, meaning that they are equal, up to subsets of arbitrarily small measures, to boundaries of convex bodies of class $C^{1,1}$; see Corollary \ref{corollary for convex bodies} below.

Our first main result is as follows.
\begin{thm}
\label{main theorem for C11 convex}
Let $f:\R^n\to\R$ be a convex function. For every measurable subset $A$ of finite Lebesgue measure in $\R^n$, and for every $\varepsilon>0$, there exists a convex function $g:\R^n\to\R$ of class $C^{1,1}$ such that
\begin{equation}
\label{eq1}
\mathcal{L}^n\left(\{x\in A:\, f(x)\neq g(x)\}\right)<\varepsilon.
\end{equation}
\end{thm}
Note that if $f=g$ in a measurable set $E$, then 
it is easy to show that
$\nabla f=\nabla g$ a.e. in $E$, so condition \eqref{eq1} is equivalent to a seemingly stronger one:
$$
\mathcal{L}^n\left(\{x\in A:\, f(x)\neq g(x) \, \textrm{ or } \nabla f(x)\neq\nabla g(x)\}\right)<\varepsilon
$$
which says that outside a set of measure less than $\varepsilon$, we have $f=g$ and $\nabla f=\nabla g$.

As a corollary we obtain a new proof of a result mentioned above.
\begin{cor}
\label{C1}
If $f:\R^n\to\R$ is convex, then for any $\varepsilon>0$ there is $g\in C^2(\R^n)$ such that $\mathcal{L}^n(\{x:\, f(x)\neq g(x)\})<\varepsilon$.
\end{cor}
 We do not claim however, that $g$ is convex.

\begin{cor}\label{main corollary for C11 convex}
Let $\Omega$ be a bounded, open and convex subset of $\R^n$. Then for every convex function $f:\Omega\to\R$ and every $\varepsilon>0$ there exists a convex function $g:\R^n\to\R$ of class $C^{1,1}$ such that  
$\mathcal{L}^n\left(\{x\in \Omega:\, f(x)\neq g(x) \}\right)<\varepsilon$.
\end{cor}

Recall that a subset $W$ of $\R^n$ is a {\em compact convex body} if $W$ is compact and convex, with nonempty interior. We will say that a compact convex body $W$ is of class $C^{1,1}$ provided that its boundary $\partial W$ is a $C^1$ hypersurface of $\R^n$ such that the outer unit normal $N:\partial W\to\mathbb{S}^{n-1}$ is a Lipschitz mapping. If $W$ is a compact convex body, this is equivalent to saying that the Minkowski functional of $W$ is of class $C^{1,1}$ on $\R^n\setminus B(0, \varepsilon)$ for some $\varepsilon>0$, or that $W$ can be locally parameterized as a graph $(x_1, ..., x_{n-1}, g(x_{1}, ..., x_{n-1}))$ (coordinates taken with respect to an appropriate permutation of the canonical basis of $\R^n$), where $g$ is of class $C^{1,1}$. 

A consequence of Theorem \ref{main theorem for C11 convex} is that the boundary of every compact convex body in $\R^n$ is of class $C^{1,1}$ up to a subset of arbitrarily small $(n-1)$-dimensional Hausdorff measure.

\begin{cor}
\label{corollary for convex bodies}
Let $W$ be a compact convex body in $\R^n$. Then for every $\varepsilon>0$ there exists a compact convex body $W_{\varepsilon}$ of class $C^{1,1}$ such that
$\mathcal{H}^{n-1}\left(\partial W\setminus \partial W_{\varepsilon}\right)< \varepsilon$.
\end{cor}

The following examples show that the assumption $\mathcal{L}^{n}(A)<\infty$ in Theorem \ref{main theorem for C11 convex} cannot be dispensed with (unless other hypotheses on the global behaviour and the growth of $f$ are put in their place).

\begin{ex}
\label{C11 does not suit unbounded domains}
Let $f:\R\to\R$ be defined by $f(x)=x^4$. Then there is no function $g\in C^{1,1}(\R)$ such that $\mathcal{L}^{1}(\{x\in\R : f(x)\neq g(x)\})<\infty$. 
\end{ex}
Indeed, the second derivative of $f$ is bounded by a constant only on a set of finite measure.
\begin{ex}
Let $f:\R^2\to\R$ be defined by $f(x,y)=|x|$. Then $f$ is the only convex function $g:\R^2\to\R$ such that $\mathcal{L}^n\left(\{x\in\R^n : f(x)\neq g(x)\}\right)<\infty$.
\end{ex}

More generally we have the following.

\begin{prop}
\label{if f is not coercive nothing can be done globally}
Let $P:\R^n\to X$ be the orthogonal projection onto a linear subspace $X$ of $\R^n$ of dimension $k$, with $1\leq k\leq n-1$, let $c:X\to\R$ be a convex function, and define $f(x)=c(P(x))$. Then $f$ is the only convex function $g:\R^n\to\R$ such that $\mathcal{L}^n\left(\{x\in\R^n : f(x)\neq g(x)\}\right)<\infty$.
\end{prop}
\begin{proof}
Let $g:\R^n\to\R$ be a convex function such that $\mathcal{L}^{n}(A)<\infty$, where
$$
A:=\{x\in\R^n : f(x)\neq g(x)\}.
$$
Let $X^{\perp}$ stand for the orthogonal complement of $X$ in $\R^n$. By Fubini's theorem, for $\mathcal{H}^k$-almost every point $x\in X$, we have that for $\mathcal{H}^{n-k-1}$-almost every direction $v\in X^{\perp}, |v|=1$, the line
$$ 
L(x,v):=\{x+tv : t\in\R\}
$$
intersects $A$ in a set of finite $1$-dimensional measure. This implies that for all such $x\in X, v\in X^{\perp}$, the set $L(x,v)\cap (\R^n\setminus A)$ contains sequences 
$$
x_j^{\pm}:=x+t_{x,j}^{\pm} v\in \R^n\setminus A, j\in\N
$$
with $\lim_{j\to\pm\infty}t_{x,j}^{\pm}=\pm\infty$. Since $f=f\circ P$, this means that
$$
f(x)=f(x+t_{x,j}^{\pm}v)=g(x+t_{x,j}^{\pm}v),
$$
and because $t\mapsto g(x+tv)$ is convex we see that
$$
f(x+tv)=f(x)=g(x+tv)
$$
for all $t\in\R$ and every such $x$, $v$. By continuity of $f$ and $g$ this implies that
$$
f(x+tv)=g(x+tv)
$$ for all $x\in X$, $v\in X^{\perp}$, and this shows that $f=g$ on $\R^n$. 
\end{proof}

In light of Example~\ref{C11 does not suit unbounded domains}, Theorem~\ref{main theorem for C11 convex} and Corollary~\ref{main corollary for C11 convex} seem the best possible results for the Lusin property of class $C^{1,1}_{\rm conv}$.

Thus it is natural to consider the Lusin property of class $C^{1,1\, {\rm loc}}_{\rm conv}$ or $C^2_{\rm conv}$, where we do not have any restrictions on the growth of the second derivative. Clearly, if the function $f$ from Proposition~\ref{if f is not coercive nothing can be done globally} is not already $C^{1,1}_{\rm loc}$, it does not have the Lusin property of class $C^{1,1\, {\rm loc}}_{\rm conv}$ or $C^2_{\rm conv}$.
Under the assumptions of Proposition~\ref{if f is not coercive nothing can be done globally}, if $0\neq v\in X^\perp$, then the function $f$ is constant along the line $t\to tv$ and hence there is no linear function $\ell:\R^n\to\R$ such that $f(x)-\ell(x)\to\infty$ as $|x|\to\infty$. That is, $f$ is not essentially coercive, where the essential coercivity is defined as follows:

\begin{defn}
\label{essential coercivity}
We say that a convex function $f:\R^n\to\R$ is {\em essentially coercive} provided there exists a linear function $\ell:\R^n\to\R$ such that 
$$
\lim_{|x|\to\infty}\left( f(x)-\ell(x)\right)=\infty.
$$
\end{defn}

It is natural to wonder whether the lack of essential coercivity of a convex function defined on $\R^n$ is the only obstruction for it to satisfy a Lusin property of type $C^{1,1 \, {\rm loc}}_{{\rm conv}}(\R^n)$ or $C^2_{\textrm{conv}}(\R^n)$. Our second main result shows that this is indeed so in the $C^{1,1 \, \textrm{loc}}_{\textrm{conv}}(\R^n)$ case.

\begin{thm}
\label{main theorem for loc C11 convex}
Let $f:\R^n\to\R$ be a convex function, and assume that $f$ is not of class $C^{1,1}_{{\rm loc}}(\R^n)$. Then $f$ is essentially coercive if and only if for every $\varepsilon>0$ there exists a convex function $g:\R^n\to\R$ of class $C^{1,1}_{{\rm loc}}(\R^n)$ such that $\mathcal{L}^{n}\left(\{x\in \R^n : f(x)\neq g(x)\}\right)<\varepsilon$. 
\end{thm}

 As an easy application we obtain the following generalization of Corollary \ref{corollary for convex bodies}. We call $S$ a convex hypersurface of $\R^n$ provided that $S$ is the boundary of a closed convex set $W$ with nonempty interior (not necessarily bounded) in $\R^n$, and we say that $S$ is of class $C^{1,1}_{\textrm{loc}}$ whenever $S$ is a 1-codimensional $C^1$ submanifold of $\R^n$ such that the outer unit normal $N:S\to\mathbb{S}^{n-1}$ is a locally Lipschitz mapping. Equivalently, the Minkowski functional $\mu_{W}$ of $W$ is of class $C^{1,1}_{\textrm{loc}}$ on $\R^n\setminus \mu_{W}^{-1}(0)$.

\begin{cor}
\label{corollary for convex hypersurfaces}
Let $S$ be a convex hypersurface of $\R^n$, and assume that $S$ is not of class $C^{1,1}_{{\rm loc}}$. Then the following assertions are equivalent:
\begin{enumerate}
\item $S$ does not contain any line.
\item For every $\varepsilon>0$ there exists a convex hypersurface $S_{\varepsilon}$ of class $C^{1,1}_{\textrm{loc}}$ of $\R^n$ such that
$\mathcal{H}^{n-1}\left(S\setminus S_{\varepsilon}\right)< \varepsilon$.
\end{enumerate}
\end{cor}

\section{Preliminaries and tools}

In this section we explain some known results and techniques which we will use in the proofs of our main results.

A convex function $f:\Omega\to\R$ defined on an open and convex set $\Omega\subset\R^n$ is locally Lipschitz and hence differentiable almost everywhere by Rademacher's theorem. In fact the following result is true: 

\begin{thm}
\label{T1}
If a convex function
$f:\Omega\to\R$ is defined on an open and convex set $\Omega\subset\R^n$, and $D\subset\Omega$ is the set of points where $f$ is differentiable, then $\nabla f|_D$ is continuous.
\end{thm}
For a proof see e.g. \cite[Theorem~IV.E]{RV}. An elementary and a straightforward  argument can also be found
in \cite[p.727]{KirchheimKristensen}.

According to Aleksandrov's theorem \cite{Alexandroff}, at almost every point $x$, where $f$ is differentiable, there is a symmetric $n\times n$ matrix $\nabla^2 f(x)$ such that
\begin{equation}
\label{eq2}
\lim_{y\to x}\frac{f(y)-f(x)-\langle \nabla f(x), y-x\rangle-\frac{1}{2}\langle\nabla^{2}f(x)(y-x), y-x\rangle}{|y-x|^2}=0.
\end{equation}
For modern proofs, see for example \cite[Theorem~7.10]{AlbertiAmbrosio}, \cite{BCP},
\cite[Theorem~A.2]{CIL}
\cite[Theorem~6.9]{EvansGariepy}.

A common technique of showing that a class of functions has a Lusin property of class $C^{k}$ is based on the Whitney extension theorem. For example, it follows from the Aleksandrov theorem that a convex function satisfies the assumptions of a $C^2$-version of the Whitney extension theorem outside a set of an arbitrarily small measure and hence the class $C_{\rm conv}(\R^n)$ has the Lusin property of class $C^2(\R^n)$, see \cite{Alberti2,EvansGangbo,Imomkulov}. Unfortunately, Whitney's construction does not preserve convexity.
Instead we will use results and techniques from \cite{Azagra2019,AzagraLeGruyerMudarra,AzagraMudarra1,AzagraMudarra2} which we next review.

\begin{thm}
$($\cite[Corollary 1.3]{AzagraMudarra1} and \cite[Theorem 2.4]{AzagraLeGruyerMudarra}$)$.
\label{C11 Whitney thm for convex}
Let $E$ be an arbitrary subset of $\R^n$. Let $f:E\to\R$, $G:E\to\R^n$ be given functions. Then there exists a convex function $F\in C^{1,1}(\R^n)$ with $F=f$ and $\nabla F=G$ on $E$ if and only if there exists a number $M>0$ such that
\begin{equation}
\label{CW11}
f(x)-f(y)-\langle G (y), x-y \rangle \geq \frac{1}{2M} | G (x)- G(y) |^{2}
\end{equation}
for all $x,y\in E$. In fact, the formula  
$$
F={\rm conv}\left(x\mapsto \inf_{y\in E}\left\{f(y)+\langle G(y), x-y\rangle+\frac{M}{2}|x-y|^2\right\}\right)
$$
defines such an extension, with the additional property that ${\rm Lip}(\nabla F)\leq M$.
\end{thm}
Here $\textrm{conv}(x\mapsto g(x))$ denotes the convex envelope of the function $g$, that is, the largest convex function $\varphi$ such that $\varphi\leq g$.

This result was first proved in \cite[Corollary 1.3]{AzagraMudarra1}, but the proof given in 
\cite[Theorem 2.4]{AzagraLeGruyerMudarra} is elementary and much simpler. The next elementary lemma shows another condition that is equivalent to \eqref{CW11}. In fact, in our proofs we will apply Theorem~\ref{C11 Whitney thm for convex} by verifying condition (b) from below.
\begin{lem}
\label{L1}
Let $E$ be an arbitrary subset of $\R^n$. Let $f:E\to\R$, $G:E\to\R^n$ be given functions and let $M>0$ be a given constant. Then the following conditions are equivalent:
\begin{itemize}
\item[(a)] $\displaystyle f(x)-f(y)-\langle G (y), x-y \rangle \geq \frac{1}{2M} | G (x)- G(y) |^{2}$ for all $x,y\in E$;
\item[(b)]
$\displaystyle f(z)+\langle G(z), x-z\rangle\leq f(y)+\langle G(y), x-y\rangle +\frac{M}{2}|x-y|^2$ for all $y,z\in E$ and all $x\in\R^n$.
\end{itemize}
\end{lem}
\begin{proof}
Implication from (a) to (b) is just \cite[Lemma 2.6]{AzagraLeGruyerMudarra}.
Thus it remains to show that (b) implies (a). Renaming variables allows us to rewrite (b) as
$$
f(y)+\langle G(y),\xi-y\rangle\leq f(x)+\langle G(x),\xi-x\rangle +\frac{M}{2}|\xi-x|^2
\quad
\text{for $x,y\in E$ and $\xi\in\R^n$,}
$$
which is equivalent to
$$
f(x)-f(y)-\langle G(y),x-y\rangle \geq \langle G(y)-G(x),\xi-x\rangle -\frac{M}{2}|\xi-x|^2,
\quad
\text{for $x,y\in E$ and $\xi\in\R^n$.}
$$
Since the inequality is true for all $\xi\in\R^n$, we can take $\xi$ such that $\xi-x=\frac{1}{M}(G(y)-G(x))$ and (a) follows.
\end{proof}

In the same spirit, a more complicated version of Theorem~\ref{C11 Whitney thm for convex}
for $C^{1,1}_{\textrm{loc}}$ convex extensions of $1$-jets has been established in \cite[Theorem~1.10]{Azagra2019}. 
However, we will need the following special case which is easier to state.
\begin{thm}
$($\cite[Theorem~1.3]{Azagra2019}$)$.
\label{Corollary from Azagra2019}
Let $E$ be an arbitrary nonempty subset of $\R^n$. Let $f:E\to\R$, $G:E\to\R^n$ be functions such that 
\begin{equation}\label{essentially coercive data corollary 2}
{\rm span}\{G(x)-G(y) : x, y\in E\}=\R^n.
\end{equation}
Then there exists a convex function $F\in C^{1,1}_{{\rm loc}}(\R^n)$ such that $F_{|_E}=f$ and $(\nabla F)_{|_E}=G$ if and only if for each $k\in \N$ there exists a number $A_k\geq 2$ such that
\begin{equation}\label{every tangent function lies above all tangent planes corollary 2}
f(z)+\langle G(z), x-z\rangle
\leq f(y)+\langle G(y), x-y\rangle +\frac{A_k}{2}|x-y|^2
\end{equation}
for every $z\in E$, $y\in E\cap B(0, k)$, $x\in B(0, 4k)$.

\end{thm}

We will also need: 

\begin{thm}
$($\cite[Theorem 1.11]{AzagraMudarra2} and \cite[Lemma 4.2]{Azagra2013}$)$.
\label{rigid global behaviour of convex functions}
For every convex function $f:\R^n\to\R$, there exist a unique linear subspace $X$ of $\R^n$, a unique vector 
$v\in X^{\perp}$, and a unique essentially coercive function $c:X\to\R$ such that $f$ can be written in the form
$$
f(x)=c(P(x)) +\langle v, x\rangle 
\quad
\text{for all $x\in\R^n$,}
$$
where $P:\R^n\to X$ is the orthogonal projection.
\end{thm}

The next extension lemma is well known, but for the sake of completeness we will provide a proof.
\begin{lem}
\label{L2}
Suppose that $f:W\to\R$ is convex and $K$-Lipschitz, where $W\subset\R^n$ is convex. Then,
\begin{equation}
\label{eq4}
\tilde{f}(x)=\inf_{z\in W} \{f(z)+K|x-z|\}, 
\quad
x\in\R^n
\end{equation}
is convex and $K$-Lipschitz on $\R^n$, and $\tilde{f}=f$ on $W$.
\end{lem}
\begin{proof}
$K$-Lipschitz continuity of $f$ on $W$ implies that for $x,z\in W$ we have
$$
f(x)\leq f(z)+K|x-z|
\quad
\text{so}
\quad
f(x)\leq\tilde{f}(x).
$$
On the other hand, $\tilde{f}(x)\leq f(x)+K|x-x|=f(x)$ so $f=\tilde{f}$ on $W$. To prove that $\tilde{f}$ is $K$-Lipschitz, let $x,y\in\R^n$ and assume that $\tilde{f}(x)\geq \tilde{f}(y)$. For $z\in W$ we have
$$
\tilde{f}(x)\le f(z)+K|x-z|\leq
(f(z)+K|y-z|)+K|x-y|,
$$
so taking infimum on the right hand side over $z\in W$ yields
$K$-Lipschitz continuity of $\tilde{f}$.
It remains to show that $\tilde{f}$ is convex. If $x,y\in\R^n$ and $\lambda\in [0,1]$, then for any $z,w\in W$ we have
\begin{equation*}
\begin{split}
\tilde{f}(\lambda x+(1-\lambda)y)
&\leq 
f(\lambda z+(1-\lambda) w)+
K\big|\big(\lambda x+(1-\lambda)y\big)-\big(\lambda z+(1-\lambda) w\big)\big|\\
&\leq 
\lambda(f(z)+K|x-z|) + (1-\lambda)(f(w)+K|y-w|)
\end{split}
\end{equation*}
and taking infimum over $z,w\in W$ yields convexity of $\tilde{f}$.
\end{proof}

\section{Proofs of the main results}

\subsection{Proof of Theorem \ref{main theorem for C11 convex}} 
Let $E\subset A$ be a compact set such that $\mathcal{L}^n(A\setminus E)<\varepsilon/2$ and that all points $x\in E$ satisfy \eqref{eq2}. It is easy to see that
$$
 E=\bigcup_{j=1}^\infty E_j,
 \qquad
 E_1\subset E_2\subset\ldots,
$$
where
$$
E_{j}=\Big\{y\in E: \, f(x)-f(y)-\langle\nabla f(y), x-y\rangle \leq j |x-y|^2 \textrm{ for all } x\in\R^n \textrm{ s.t. } |x-y|\leq \frac{1}{j}\Big\}.
$$
Since by Theorem~\ref{T1}, $\nabla f$ is continuous on $E$, it is easy to check that the sets $E_j$ are closed and hence compact. 
Since $\mathcal{L}^n(E)<\infty$, $\mathcal{L}^n(E\setminus E_N)<\varepsilon/2$ for some $N$, so
$\mathcal{L}^n(A\setminus E_N)<\varepsilon$. Thus
\begin{equation}
\label{eq3}
\forall y\in E_N\ \forall x\in\R^n \ \left(|x-y|\leq N^{-1}\Rightarrow f(x)-f(y)-\langle\nabla f(y), x-y\rangle \leq N |x-y|^2\right).
\end{equation}
Since $E_N$ is compact, $E_N\subset \overline{B}(0,r)$ for some $r>0$. Let $R>r+N^{-1}$ and let $W=\overline{B}(0,R)$.
Since convex functions are locally Lipschitz, $f|_W$ is $K$-Lipschitz for some $K>0$ and we may assume that $K\geq 1$.

Let $\tilde{f}$ be defined by \eqref{eq4}. Then $\tilde{f}$ is convex and $K$-Lipschitz on $\R^n$. Since $\tilde{f}=f$ in $W$, the function $\tilde{f}$ satisfies \eqref{eq3}. Observe that
\begin{equation}
\label{eq5}    
\tilde{f}(x)-\tilde{f}(y)-\langle\nabla \tilde{f}(y), x-y\rangle \leq 2KN |x-y|^2
\quad
\text{for all $y\in E_N$ and all $x\in\R^n$.}
\end{equation}
Indeed, if $|x-y|\leq N^{-1}$, \eqref{eq5} follows from \eqref{eq3}. If $|x-y|>N^{-1}$, then
$$
\tilde{f}(x)-\tilde{f}(y)-\langle\nabla \tilde{f}(y), x-y\rangle \leq
K|x-y|+|\nabla\tilde{f}(y)|\, |x-y|\leq 2K|x-y|<2KN|x-y|^2.
$$
Also convexity of $\tilde{f}$ yields
\begin{equation}
\label{eq6} 
\tilde{f}(z)+\langle\nabla\tilde{f}(z),x-z\rangle\leq \tilde{f}(x)
\quad
\text{for all $z\in E_N$ and all $x\in\R^n$.}
\end{equation}
By combining \eqref{eq6} and \eqref{eq5} we get
$$
\tilde{f}(z)+\langle\nabla \tilde{f}(z),x-z\rangle \leq
\tilde{f}(y)+\langle\nabla \tilde{f}(y),x-y\rangle +2KN|x-y|^2
\quad
\text{for $y,z\in E_N$ and $x\in \R^n$.}
$$
This is condition (b) from Lemma~\ref{L1}. Since the condition is equivalent to \eqref{CW11}, Theorem~\ref{C11 Whitney thm for convex} gives that there is a $C^{1,1}$ function $F$ defined in $\R^n$ such that $F=\tilde{f}=f$ and $\nabla F=\nabla\tilde{f}=\nabla f$ in $E_N$. This and the fact that $\mathcal{L}^n(A\setminus E_N)<\varepsilon$ complete the proof.
\qed

\subsection{Proof of Corollary \ref{C1}}
According to Theorem~\ref{main theorem for C11 convex}, for each positive integer $i$, there is $g_i\in C^{1,1}(\R^n)$ such that $\mathcal{L}^n(\{x\in B(0,i):\, f(x)\neq g_i(x)\})<\varepsilon/2^{i+1}$. Let $\{\varphi_i\}_{i=1}^\infty$
be a smooth partition of unity subordinate to the covering $\{B(0,i)\}_{i=1}^\infty$ of $\R^n$. Then 
$g=\sum_{i=1}^\infty g_i\varphi_i\in C^{1,1}_{\rm loc}$ satisfies $\mathcal{L}^n(\{x\in\R^n:\, g(x)\neq f(x)\}<\varepsilon/2$ and the result follows from Whitney's theorem \cite[Theorem~4]{Whitney2} according to which a $C^{1,1}_{\rm loc}$ function on $\R^n$ coincides with a $C^2$ function outside a set of measure less than $\varepsilon/2$.
\qed

\subsection{Proof of Corollary \ref{main corollary for C11 convex}}

Take a compact convex body $W_{\varepsilon}$ such that
\begin{equation}
\label{K and Omega}
W_{\varepsilon}\subset\Omega,  \, \textrm{ and } \mathcal{L}^{n}\left(\Omega\setminus W_{\varepsilon}\right)<\frac{\varepsilon}{2}.
\end{equation}
Denote the Lipschitz constant of $f|_{W_\varepsilon}$ by $K$ (notice that $W_{\varepsilon}$ is at positive distance from the boundary of $\Omega$, so this Lipschitz constant exists).
According to Lemma~\ref{L2},
$$
\tilde{f}(x)=\inf_{z\in W_{\varepsilon}}\{f(z)+K|x-z|\}
$$
is convex on $\R^n$ and $\tilde{f}=f$ on $W_\varepsilon$.
Therefore, we can apply Theorem~\ref{main theorem for C11 convex} to $\tilde{f}$ and find a function $g\in C^{1,1}_{\textrm{conv}}(\R^n)$ such that 
$$
\mathcal{L}^{n}\big(\{x\in W_\varepsilon :\, g(x)\neq f(x)\}\big)=
\mathcal{L}^{n}\big(\{x\in W_\varepsilon :\, g(x)\neq \tilde{f}(x)\}\big)<\frac{\varepsilon}{2}.
$$
This and \eqref{K and Omega} imply that
$\mathcal{L}^{n}\left(\{x\in\Omega : g(x)\neq f(x)\}\right)<\varepsilon$. \qed

\subsection{Proof of Corollary \ref{corollary for convex bodies}}

We can assume that $0\in \textrm{int}(W)$ and consider the Minkowski functional of $W$, defined by
$$
\mu(x)=\inf\{\lambda\geq 0 \, : \, x\in \lambda W\},
$$
which is a Lipschitz convex function on $\R^n$. Let $L$ be the Lipschitz constant of $\mu$. By using Corollary~\ref{main corollary for C11 convex} we may find a function $g\in C^{1,1}_{\textrm{conv}}(\R^n)$ such that
$$
\mathcal{L}^{n}\left( \{x\in 2 W \, : \, \mu(x)\neq g(x)\}\right)<\frac{\varepsilon}{L}.
$$
Now consider the annulus 
$$
C_{1,2}:=2W\setminus W=\{x\in\R^n : 1< \mu(x)\leq 2\},
$$
and define 
$$
A=\{x\in C_{1,2} : \mu(x)\neq g(x)\}.
$$ 
By the coarea formula for Lipschitz functions (see \cite[Theorem 3.10]{EvansGariepy} for instance) we have
$$
\varepsilon> L\, \mathcal{L}^{n}(A)\geq\int_{A}|\nabla \mu(x)|\, dx=\int_{1}^{2}\mathcal{H}^{n-1}\left(A\cap \mu^{-1}(t)\right)\, dt.
$$
This inequality and Sard's theorem imply that there exists  a regular value $t_{0}\in (1, 2)$ of $g\in C^{1,1}$ such that
$$
\mathcal{H}^{n-1}\left(A\cap \mu^{-1}(t_0)\right)< \varepsilon.
$$
Then we can define 
$$
W_{\varepsilon}=\frac{1}{t_0}g^{-1}(-\infty, t_0],
$$
so that $W_{\varepsilon}$ is a convex body of class $C^{1,1}$, with boundary
$$
\partial W_{\varepsilon}=\frac{1}{t_0}g^{-1}(t_0)
\quad
\text{and hence}
\quad
t_0(\partial W\setminus \partial W_\varepsilon)=A\cap \mu^{-1}(t_0).
$$
This yields
$$
\mathcal{H}^{n-1}(\partial W\setminus\partial W_\varepsilon)\leq 
t_{0}^{n-1} \mathcal{H}^{n-1}\left(\partial W\setminus \partial W_{\varepsilon}\right)=
\mathcal{H}^{n-1}\left(A\cap \mu^{-1}(t_0)\right)<\varepsilon.
$$
\qed

\subsection{Proof of Theorem \ref{main theorem for loc C11 convex}.}

The necessity of the essential coercitivity assumption is clear from Proposition~\ref{if f is not coercive nothing can be done globally} and Theorem~\ref{rigid global behaviour of convex functions}. 
The fact that this assumption is sufficient will follow from Theorem~\ref{Corollary from Azagra2019}.
The rest of the proof is similar to that of Theorem~\ref{main theorem for C11 convex} but more complicated.

Thus assume that $f:\R^n\to \R$ is convex, $f\not\in C^{1,1}_{\rm loc}$ and $f$ is essentially coercive. Fix $0<\varepsilon<1$. It remains to show that there is a convex $g\in C^{1,1}_{\rm loc}$ such that $\mathcal{L}^n(\{f\neq g\})<\varepsilon$.

By Aleksandrov's theorem and by Theorem~\ref{T1}\footnote{We could use Lusin's theorem instead of Theorem~\ref{T1}.}, there is a closed set $A\subset\R^n$ such that
$$
\mathcal{L}^n(\R^n\setminus A)<\frac{\varepsilon}{2},
\quad
\text{$\nabla f|_A$ is continuous,}
\quad
\text{and}
\quad
\text{\eqref{eq2} is true for all $x\in A$.}
$$
As in the proof of Theorem~\ref{main theorem for C11 convex}, 
$$
 A=\bigcup_{j=1}^\infty E_j,
 \qquad
 E_1\subset E_2\subset\ldots,
$$
where
$$
E_{j}=\Big\{y\in A: \, f(x)-f(y)-\langle\nabla f(y), x-y\rangle \leq j |x-y|^2 \textrm{ for all } x\in\R^n \textrm{ s.t. } |x-y|\leq \frac{1}{j}\Big\}.
$$
Since $\nabla f$ is continuous on $A$, it easily follows that the sets $E_j$ are closed and hence measurable.

We set $B_0=\emptyset$, and for each $k\in\N$, we define
$$
B_k:=B(0, k), \, \textrm{ and } \, A_k:=A\cap(B_k\setminus B_{k-1}).
$$
Now, for each $k\in\N$, since the sequence $\{E_j\}_{j\in\N}$ is increasing and $A_k=\bigcup_{j=1}^{\infty} \left(E_j\cap A_k\right)$, we can find $j_{k}\in\N$ such that
$$
\mathcal{L}^{n}(A_k\setminus E_{j_k})<\frac{\varepsilon}{2^{k+1}},
$$
and define, for each $k\in\N$,
$$
C_k:=E_{j_k}\cap A_k, 
$$
and
$$
C:=\bigcup_{k=1}^{\infty}C_k.
$$
We may obviously assume that 
\begin{equation}
j_{k}\leq j_{k+1} \textrm{ for all } k\in\N.
\end{equation}
We then have that 
\begin{equation}
\label{estimate for the measure of A minus C}
\mathcal{L}^{n}(A\setminus C)=\sum_{k=1}^{\infty}\mathcal{L}^{n}(A_k\setminus C_k)<\frac{\varepsilon}{2}
\quad
\text{so}
\quad
\mathcal{L}^n(\R^n\setminus C)<\varepsilon.
\end{equation}

\begin{lem}\label{properties of Ej again}
For each $k\in\N$ there exists a number $\beta_{k}\geq 1$ such that
\begin{equation}\label{estimates for applying locally C11 convex WET}
f(x)-f(y)-\langle\nabla f(y), x-y\rangle\leq \beta_k |x-y|^2 \textrm{ for all } y\in C\cap B_k \textrm{ and all } x\in B_{4k}.
\end{equation}
\end{lem}
\begin{proof}
Take $y\in C\cap B_k$
and note that since $(j_k)$ is increasing we have $C\cap B_k\subset E_{j_k}\cap B_k\subset B_{4k}$. 
In particular $y\in E_{j_k}$.

If $x\in\R^n$ is such that $|x-y|\leq 1/j_k$, the inequality we seek obviously holds with $\beta_k= j_k$, because of the definition of $E_{j_k}$. On the other hand, if $|x-y|>1/j_k$ and $x\in B_{4k}$, then, since $f$ is Lipschitz on the ball $B_{4k}$, we have
$$
f(x)-f(y)-\langle\nabla f(y), x-y\rangle\leq 
2\,\textrm{Lip}\left(f_{|_{B_{4k}}}\right)|x-y|\leq 2\,\textrm{Lip}\left(f_{|_{B_{4k}}}\right) j_k |x-y|^2.
$$
In any case the Lemma is satisfied with $\beta_k=\max\left\{j_k, \, 2j_k\,\textrm{Lip}\left(f_{|_{B_{4k}}}\right)\right\}$. 
\end{proof}

Since $f$ is convex we have, for all $z\in C$, $x\in\R^n$, that
\begin{equation}
f(z)+\langle \nabla f(z), x-z\rangle\leq f(x),
\end{equation}
which combined with the preceding lemma gives us
$$
f(z)+\langle \nabla f(z), x-z\rangle \leq
f(y)+\langle\nabla f(y), x-y\rangle+\beta_k |x-y|^2 
$$
for all $z\in C$, $y\in C\cap B_k$, $x\in B_{4k}$.
That is to say, the jet $(f(y), \nabla f(y))$, $y\in C$, satisfies condition \eqref{every tangent function lies above all tangent planes corollary 2} of Theorem~\ref{Corollary from Azagra2019} with $A_k=2\beta_k\geq 2$.

Finally, let us check condition \eqref{essentially coercive data corollary 2} which in our case reads as
$$
\operatorname{span}\{\nabla f(x)-\nabla f(y):\, x,y\in C\}=\R^n.
$$
Fix some $y_0\in C$, and consider the function $g(x)=f(x)-f(y_0)-\langle\nabla f(y_0), x-y_0\rangle$. Since $f$ is essentially coercive and convex, $g$ is also essentially coercive and convex. 
We have that 
$$
\operatorname{span}\{\nabla f(x)-\nabla f(y) : x, y\in C\}=
\operatorname{span}\{\nabla f(x)-\nabla f(y_0) : x\in C\}=\operatorname{span}\{\nabla g(x) : x\in C\}.
$$ 
Thus it suffices to show that $Y:=\textrm{span}\{\nabla g(x) : x\in C\}=\R^n$.
Seeking a contradiction, suppose that $Y\neq\R^n$. We can then take a vector $0\neq v\in Y^{\perp}$ such that
\begin{equation}
\label{if the derivatives of g do not span}
\langle \nabla g(x), v\rangle =0 \textrm{ for all } x\in C.
\end{equation}
Since $\mathcal{L}^{n}(\R^n\setminus C)<\varepsilon$, an easy application of Fubini's theorem shows that there exists $x_0$ perpendicular to $v$ such that the intersection of the line $L:=\{x_0+tv :t\in\R\}$ with the set $\R^n\setminus C$ has finite one-dimensional measure. This implies that $L\cap C$ must contain sequences
$x_j^{\pm}:=x_0+t_{j}^{\pm} v, j\in\N$
with $\lim_{j\to\pm\infty}t_{j}^{\pm}=\pm\infty$.

Consider the restriction of $g$ to the line $L$ i.e., consider the convex function $h(t)=g(x_0+tv)$. Since by \eqref{if the derivatives of g do not span},
$$
h'(t_j^\pm)=\langle\nabla g(x_j^\pm),v\rangle =0,
$$
it follows that $h$ is constant and hence $g$ is constant on the line $L$.
But this contradicts the fact that $g$ is essentially coercive.

We have thus checked that the $1$-jet $(f(y), \nabla f(y))$, $y\in C$, satisfies all the conditions of Theorem \ref{Corollary from Azagra2019}, and therefore there exists a locally $C^{1,1}$ convex function $F:\R^n\to\R$ such that $F=f$ on $C$, and also $\nabla F=\nabla f$ on $C$. In particular we have that
$$
\mathcal{L}^n\left(\left\{ x\in \R^n : \, f(x)\neq F(x)\right\}\right) \leq\mathcal{L}^{n}\left(\R^n\setminus C\right) < \varepsilon.
$$
The proof of Theorem \ref{main theorem for loc C11 convex} is complete.
\qed

\subsection{Proof of Corollary \ref{corollary for convex hypersurfaces}}
We will need to use the following:
\begin{lem}
Let $W$ be a closed convex set such that $0\in\textrm{int}(W)$, and $\mu=\mu_{W}$ denote the Minkowski functional of $W$. The following assertions are equivalent:
\begin{enumerate}
\item[{(a)}] $W$ does not contain any line. 
\item[{(b)}] $\partial W$ does not contain any line.
\item[{(c)}] $\mu^{-1}(0)$ does not contain any line
\item[{(d)}] $\mu$ is essentially coercive.
\end{enumerate}
\end{lem}
\begin{proof}
$(a)\implies (b)$ is obvious.

$(b)\implies (c)$: if for some $x, v$ with $v\neq 0$ we have $\mu(x+tv)=0$ for all $t\in\R$ then, for any $y\in \mu^{-1}(1)=\partial W$ we have $\mu(y+tv)=\mu(y-x+x+tv)\leq\mu(y-x)+\mu(x+tv)=\mu(y-x)$ for all $t\in\R$. In particular the convex function $\R\ni t\mapsto \mu(y+tv)\in\R$ is bounded above, hence it is constant. That is to say, $\mu(y+tv)=\mu(y)=1$ for all $t\in\R$, which means that $\partial W$ contains the line $\{y+tv : t\in\R\}$.

$(c)\implies (d)$: If $\mu$ is not essentially coercive then by Theorem~\ref{rigid global behaviour of convex functions}, for some $w\neq 0$ we have that $t\mapsto \mu(tw)$ is linear, and this may only happen if $\mu(tw)=0$ for all $t\in\R$, because $\mu\geq 0$.

$(d)\implies (a)$. If $W$ contains a line $\{x+tv: t\in\R\}$, then $t\mapsto \mu(x+tv)$ convex and bounded from above by $1$ so is is constant. This prevents $\mu$ from being essentially coercive.
\end{proof}
Now we can prove Corollary \ref{corollary for convex hypersurfaces}.

$(1)\implies (2)$: Let $W$ be the closed convex set with nonempty interior such that $S=\partial W$. Without loss of generality we may assume that $0$ is an interior point of $W$.  Denote the Minkowski functional of $W$ by $\mu$. 
By the assumption $(1)$ and the preceding lemma, $\mu$ is essentially coercive. Also observe that $\mu$ is $L$-Lipschitz, where $1/L=d(0, S)$. Then the same proof as in Corollary \ref{corollary for convex bodies} (replacing $C^{1,1}$ with $C^{1,1}_{\textrm{loc}}$ at appropriate places, and using Theorem \ref{main theorem for loc C11 convex} instead of Corollary \ref{main corollary for C11 convex}) shows $(2)$.

$(2)\implies (1)$:
Suppose to the contrary that $S$ contains a line, or equivalently that $\mu^{-1}(0)$ contains a line $L_0=\{tv:\, t\in\R\}$, $v\neq 0$.

Given $\varepsilon>0$, let $S_\varepsilon$ be a convex hypersurface of class $C^{1,1}_{\rm loc}$ such that $\mathcal{H}^{n-1}(S\setminus S_\varepsilon)<\varepsilon$. Let $W_\varepsilon$ be the closed convex set such that $\partial W_\varepsilon=S_\varepsilon$. If $\varepsilon>0$ is small enough, we may assume that $0\in{\rm int}(W)$ and $0\in{\rm int}(W_\varepsilon)$.

Indeed, since $S\not\in C^{1,1}_{\rm loc}$, it is not a flat hyperplane and we can find points $p_1,\ldots,p_{n+1}\in S$ such that the simplex 
$\operatorname{conv}\{p_1,\ldots,p_{n+1}\}\subset W$ has nonempty interior. By translating the coordinate system we may assume that $0$ belongs to the interior of that simplex and hence $0\in\operatorname{int}(W)$.
Then, if $\varepsilon>0$ is sufficiently small, we can find points $p_1',\ldots,p_{n+1}'\in S_\varepsilon$ so close to the points $p_1,\ldots,p_{n+1}$ that $0$ belongs to the interior of the simplex 
$\operatorname{conv}\{p_1',\ldots,p_{n+1}'\}\subset W_\varepsilon$
and hence $0\in\operatorname{int}(W_\varepsilon)$.
Denote the Minkowski functional of $W_\varepsilon$ by $\mu_\varepsilon$.

Since $\mu$ is convex, if $\nabla\mu(x)=0$, then $\mu$ attains minimum at $x$. Since $\mu\geq 0=\mu(0)$, it follows that $x\in \mu^{-1}(0)$. Therefore, $|\nabla\mu|>0$ almost everywhere in the set $\R^n\setminus \mu^{-1}(0)$.

Recall that $L_0\subset\mu^{-1}(0)$ so $\mu(tv)=0$ for all $t\in\R$. This implies that $\mu$ is constant on every line parallel to $L_0$. Indeed,
$$
\mu(x+tv)\leq\mu(x)+\mu(tv)=\mu(x)
$$
so the convex function $t\mapsto (x+tv)$ is constant as bounded from above.

Let $X=L_0^\perp$ be the orthogonal complement of $L_0$ and let $P:\R^n\to X$ be the orthogonal projection. Since $\mu$ is constant on every line parallel to $L_0$, $\mu(x)=\mu(P(x))$, hence $\nabla\mu(x)=\nabla\mu(P(x))$, and also 
\begin{equation}
\label{cylindrical structure of mu-1(0)}
\R^n\setminus\mu^{-1}(0)=P^{-1}(X\setminus \mu^{-1}(0)).
\end{equation}
Recall that $|\nabla\mu(y)|>0$ exists and is positive for almost all $y\in X\setminus\mu^{-1}(0)$.

Suppose that $E\subset\R^n\setminus\mu^{-1}(0)$ is measurable and
$$
\int_E|\nabla\mu|<\infty.
$$
Since $|\nabla\mu|$ is well defined and constant along almost all lines $P^{-1}(y)$, $y\in X\setminus\mu^{-1}(0)$ parallel to $L_0$, we note that  $P(E)\subseteq X\setminus\mu^{-1}(0)$, and apply Fubini's theorem to obtain
$$
\int_{P(E)} |\nabla\mu(y)|\mathcal{L}^1(P^{-1}(y)\cap E)\, d\mathcal{L}^{n-1}(y)=
\int_E |\nabla\mu|\, d\mathcal{L}^n<\infty.
$$
Since $|\nabla\mu(y)|>0$ for almost all $y\in X\setminus\mu^{-1}(0)$, $\mathcal{L}^1(P^{-1}(y)\cap E)<\infty$ for almost all $y\in X\setminus\mu^{-1}(0)$ i.e., almost every line parallel to $L_0$ and disjoint from $\mu^{-1}(0)$
intersects $E$ along a set of finite length. We shall use this observation below.

Let $C_a:=\{ty:\, y\in S\cap S_\varepsilon,\ t\in [0,a]\}$, $a>0$. Then for $0<t\leq a$,
$$
\mathcal{H}^{n-1}(\mu^{-1}(t)\setminus C_a)=
\mathcal{H}^{n-1}(t(S\setminus S_\varepsilon))=
t^{n-1}\mathcal{H}^{n-1}(S\setminus S_\varepsilon)<\varepsilon t^{n-1},
$$
and hence the coarea formula yields
$$
\int_{\mu^{-1}((0,a])\setminus C_a} |\nabla\mu(x)|dx= \int_{0}^{a}\mathcal{H}^{n-1}(\mu^{-1}(t)\setminus C_a)dt<\int_{0}^{a}\varepsilon t^{n-1}\, dt=\frac{\varepsilon a^{n}}{n}<\infty.
$$
It follows from an observation that we made earlier, that almost every line parallel to $L_0$ intersects $\mu^{-1}((0,a])\setminus C_a$ along a set of finite length and hence almost every line parallel to $L_0$ that is contained in $\mu^{-1}(0,a]$ intersects $\R^n\setminus C_a$ along a set of finite length. Therefore, for such a line, the set $L\cap C_a$ contains sequences
$$
x_j^{\pm}:=x+t_{x,j}^{\pm} v\in C_a, 
\quad
\text{with}
\quad
\lim_{j\to\pm\infty}t_{x,j}^{\pm}=\pm\infty.
$$
By convexity of $\mu_\varepsilon$, the fact that $\mu=\mu_\varepsilon$ on $C_a$, and that $\mu$ is constant on $L$, it follows that $\mu(y)=\mu_{\varepsilon}(y)$ for all $y\in L$, and by continuity it follows that $\mu=\mu_\varepsilon$ on every line parallel to $L_0$ and contained in $\mu^{-1}(0,a]$.
Since $a>0$ is arbitrary, it follows that $\mu=\mu_{\varepsilon}$ on every line parallel to $L_0$  and disjoint from $\mu^{-1}(0)$ (because every line $L$ parallel to $L_0$  and disjoint from $\mu^{-1}(0)$ is contained in $\mu^{-1}(0,a]$ for some $a>0$; indeed, if $\mu(x)=a>0$ for some $x\in L$, then $\mu=a$ on $L$ and $L\subset\mu^{-1}(0,a]$). Bearing in mind \eqref{cylindrical structure of mu-1(0)}, we deduce  that $\mu=\mu_{\varepsilon}$ on $\R^n\setminus\mu^{-1}(0)$, and therefore $S=\mu^{-1}(1)=\mu_{\varepsilon}^{-1}(1)=S_{\varepsilon}$. But this is absurd, because $S_{\varepsilon}$ is of class $C^{1,1}_{\textrm{loc}}$ and $S$ is not.


\begin{thebibliography}{99}



\bibitem{Alberti}
G. Alberti, {\em A Lusin type theorem for gradients}, J. Funct. Anal. 100 (1991), no. 1, 110--118. 

\bibitem{Alberti2}
G. Alberti, {\em On the structure of singular sets of convex functions},
Calc. Var. Partial Differential Equations 2 (1994), no. 1, 17--27. 

\bibitem{AlbertiAmbrosio}
G. Alberti and L. Ambrosio, {\em A geometrical approach to monotone functions in $\R^n$}, 
Math. Z. 230 (1999), no. 2, 259--316. 

\bibitem{Alexandroff}
A.D. Alexandroff, {\em Almost everywhere existence of the second differential of a convex function and some properties of convex surfaces connected with it.} (Russian)
Leningrad State Univ. Annals [Uchenye Zapiski] Math. Ser. 6, (1939). 3--35. 

\bibitem{Azagra2013} D. Azagra, {\em Global and fine approximation of convex functions}, Proc. London Math. Soc. 107 (2013), 799--824.

\bibitem{Azagra2019} D. Azagra, {\em Locally $C^{1,1}$ convex extensions of $1$-jets}, preprint, 2019, arXiv:1905.02127. To appear in Rev. Matem\'atica Iberoamericana.

\bibitem{AzagraLeGruyerMudarra}
D. Azagra, E. Le Gruyer, C. Mudarra, {\em Explicit formulas for $C^{1,1}$ and $C^{1,\omega}_{\textrm{conv}}$ extensions of 1-jets in Hilbert and superreflexive spaces},  J. Funct. Anal. 274 (2018), 3003-3032.

\bibitem{AzagraMudarra1} D. Azagra and C. Mudarra, {\em Whitney extension theorems for convex functions of the classes $C^1$ and $C^{1, \omega}$}, Proc. London Math. Soc. 114 (2017), 133--158.

\bibitem{AzagraMudarra2} D. Azagra and C. Mudarra, {\em Global geometry and $C^1$ convex extensions of $1$-jets}, Analysis and PDE  12 (2019) no. 4, 1065-1099.


\bibitem{BagbyZiemer} 
T. Bagby and W.P. Ziemer, {\em Pointwise differentiability and absolute continuity}, Trans. Amer. Math. Soc. 191 (1974), 129--148.

\bibitem{BCP}
G. Bianchi, A. Colesanti, C. Pucci,
{\em On the second differentiability of convex surfaces.}
Geom. Dedicata 60 (1996), 39--48.

\bibitem{BojarskiHajlasz}
B. Bojarski  and P. Haj\l asz, {\em Pointwise inequalities for Sobolev functions}, Studia Math. 106 (1993), 77--92.

\bibitem{BojarskiHajlaszStrzelecki}
B. Bojarski, P. Haj\l asz, and P. Strzelecki, {\em  Improved $C^{k, \lambda}$ approximation of higher order Sobolev functions in norm and capacity}. Indiana Univ. Math. J. 51 (2002), 507--540.

\bibitem{BourKoKris2}
J. Bourgain, M. V. Korobkov and J. Kristensen, {\em On the Morse-Sard property and level sets of $W^{n,1}$ Sobolev functions on $\R^n$},
J. Reine Angew. Math. 700 (2015), 93--112.

\bibitem{CalderonZygmund}
A. P. Calder\'on and A. Zygmund, {\em Local properties of solutions of elliptic partial differential equations}, Studia Math. 20 (1961), 171--225. 

\bibitem{CIL}
M.G. Crandall, H. Ishii, P.-L. Lions,
{\em User's guide to viscosity solutions of second order partial differential equations.}
Bull. Amer. Math. Soc. (N.S.) 27 (1992),  1--67.

\bibitem{EvansGangbo}
L.C. Evans, W. Gangbo,{\em Differential equations methods for the Monge-Kantorovich mass transfer problem.} Mem. Amer. Math. Soc. 137 (1999), no. 653.

\bibitem{EvansGariepy}
L.C. Evans, R.F. Gariepy, {\em Measure theory and fine properties of functions}. Revised edition. Textbooks in Mathematics. CRC Press, Boca Raton, FL, 2015. 

\bibitem{Federer}
H. Federer, {\em Surface area. II}, Trans. Amer. Math. Soc. 55, (1944), 438--456. 

\bibitem{Francos}
G. Francos, {\em The Luzin theorem for higher-order derivatives.} Michigan Math. J. 61 (2012),  507--516. 

\bibitem{Imomkulov}
S.A. Imomkulov, {\em Twice differentiability of subharmonic
  functions}. (Russian) Izv. Ross. Akad. Nauk Ser. Mat. 56 (1992),
877--888; translation in Russian Acad. Sci. Izv. Math. 41 (1993),
157--167.

\bibitem{Isakov}
N.M.Isakov,  {\em A global property of approximately differentiable functions},  Mathematical Notes of the Academy of Sciences of the USSR (1987) 41 (1987), 280-285.

\bibitem{KirchheimKristensen}
B. Kirchheim, J. Kristensen, {\em Differentiability of convex envelopes}. C. R. Acad. Sci. Paris S\'er. I Math. 333 (2001), no. 8, 725--728.

\bibitem{Liu1977}
Fon-Che Liu, {\em A Luzin type property of Sobolev functions}, Indiana Univ. Math. J. 26 (1977), 645--651.

\bibitem{LiuTai}
F.-C. Liu and W.-S. Tai, {\em Approximate Taylor polynomials and differentiation of functions},  Topol. Methods Nonlinear Anal. 3 (1994), no. 1, 189--196.


\bibitem{MP}
L. Moonens, W. F. Pfeffer,
{\em The multidimensional Luzin theorem.}
J. Math. Anal. Appl. 339 (2008), 746--752.

\bibitem{MichaelZiemer}
J. Michael  and W.P. Ziemer, {\em A Lusin type approximation of
  Sobolev functions by smooth functions}, Contemp. Math. 42 (1985),
135--167.


\bibitem{RV}
A.W. Roberts, D.E. Varberg,
{\em Convex functions.}
Pure and Applied Mathematics, Vol. 57. Academic Press, New York-London, 1973.
  

\bibitem{Whitney}
H. Whitney, {\em Analytic extensions of differentiable functions defined in closed sets}, Trans. Amer. Math. Soc. 36 (1934), 63--89.

\bibitem{Whitney2}
H. Whitney, {\em On totally differentiable and smooth functions},
Pacific J. Math. 1, (1951). 143--159. 

\bibitem{Ziemer}
W. P. Ziemer, {\em Weakly Differentiable Functions}, Springer--Verlag, 1989.


\end{thebibliography}
\end{document}